\documentclass{svproc}
\usepackage{amsmath, verbatim, amsfonts, amssymb, color}
\usepackage{mathrsfs}
\usepackage{dsfont}

\renewcommand\labelenumi{\textup{\alph{enumi})}}
\renewcommand\theenumi\labelenumi

 \makeatletter
\def\@makefnmark{\hbox{(\@textsuperscript{\normalfont\@thefnmark})}}
\makeatother

\allowdisplaybreaks

\newcommand{\bone}{\mathbf{1}}

\newcommand{\N}{\mathbb{N}}
\newcommand{\R}{\mathbb{R}}

\newcommand{\dl}{\delta}
\newcommand{\wg}{\wedge}
\newcommand{\eps}{\varepsilon}
\newcommand{\zt}{\zeta}

\newcommand{\cl}[1]{\mkern 1.5mu\overline{\mkern-1.5mu#1\mkern-1.5mu}\mkern 1.5mu}

\newcommand{\relmiddle}[1]{\mathrel{}\middle#1\mathrel{}}

\makeatletter
\@namedef{subjclassname@2020}{%
  \textup{2020} Mathematics Subject Classification}
\makeatother

\providecommand{\ack}[1]{\par\addvspace\baselineskip
\noindent\ackname\enspace\ignorespaces#1}%
\def\subjclassname{\textup{2020} \textit{Mathematics Subject Classification:}}%
\providecommand{\subjclass}[1]{\par\addvspace\baselineskip
\noindent\subjclassname\enspace\ignorespaces#1}%

\begin{document}
\mainmatter
\title{\bfseries H\"{o}lder estimates for resolvents of time-changed Brownian motions}
\titlerunning{H\"{o}lder estimates for resolvents of time-changed Brownian motions}

\author{Kouhei Matsuura}
\authorrunning{K.~Matsuura}

\tocauthor{David Berger, Franziska K\"uhn and Ren\'e L. Schilling}

\institute{
Institute of Mathematics, University of Tsukuba,\\
1-1-1, Tennodai, Tsukuba, Ibaraki, 305-8571, Japan\\
\email{kmatsuura@math.tsukuba.ac.jp}  }

\maketitle
\begin{abstract} 
In this paper, we study time changes of Brownian motions by positive continuous additive functionals. 
Under a certain regularity condition on the associated Revuz measures, we prove that the resolvents of the time-changed Brownian motions are locally
H\"{o}lder continuous in the spatial components. We also obtain lower bounds for the indices of the H\"{o}lder continuity. 
\ack{The author expresses his gratitude to Professor Yuichi Shiozawa for very careful reading of an earlier manuscript.
This work was supported by JSPS KAKENHI Grant number 20K22299.}
\end{abstract}

\bigskip\noindent

\section{Introduction}\label{sec:intro}
Let $B=(\{B_t\}_{t \ge 0},\{P_x\}_{x \in \mathbb{R}^d})$ be a Brownian motion on the $d$-dimensional Euclidean space $\R^d$. Let $A=\{A_t\}_{t \ge 0}$ be a positive continuous additive functional (PCAF in abbreviation) of $B$. Then, the time-changed Brownian motion $\check{B}=(\{\check{B}_t\}_{t \ge 0}, \{\check{P}_x\}_{x \in F})$ by the PCAF $A$ is defined as
\begin{align*}
\check{B}_t&=B_{\tau_t},\ \check{P}_x=P_x,\quad (t ,x) \in [0,\infty) \times F.
\end{align*}
Here, we denote by $\{\tau_t\}_{t \ge 0}$ the right continuous inverse of $\{A_t\}_{t \ge 0}$, and $F$ stands for the support of $A$ (see \eqref{eq:support} below for the definition). From the Revuz correspondence (see \eqref{eq:revuz}), the PCAF $A$ induces a Borel measure $\mu$ on $\R^d$, which is called the Revuz measure of $A$. It is known that the time-changed Brownian motion $\check{B}$ becomes a $\mu$-symmetric right process on $F$ (see, e.g. \cite[Theorem~5.2.1]{CF}).
On account of this fact, in what follows, we use the symbol $B^\mu$ ($A^\mu$ and $F^\mu$, respectively) to denote $\check{B}$ ($A$ and $F$, respectively).

A typical example of Revuz measures is of the form $f\,dm$. Here, $f\colon \R^d\to [0,\infty)$ is a locally bounded Borel measurable function, and $m$ stands for the Lebesgue measure on $\R^d$. Then, we have $A_t^{f\,dm}=\int_{0}^{t}f(B_s)\,ds$, $t \ge 0.$ However, a Revuz measure $\mu$ can be singular with respect to $m$. Then, the behavior of $B^\mu$ would be quite different from that of the standard Brownian motion. Nevertheless, if $F^\mu=\R^d$, we can simply describe the Dirichlet form $(\mathcal{E},\mathcal{F}^\mu)$ of $B^\mu$ by using the extended Dirichlet space $H_{\text{\rm e}}^{1}(\mathbb{R}^d)$ of $B$. %We refer the reader to \cite[Section~1.5]{FOT} and \cite[Definition~1.1.4]{CF} for the definition of the extended Dirichlet space of a symmetric Hunt process. 
If $d \in \{1,2\}$, we see from \cite[Theorem~2.2.13]{CF} that $H_{\text{\rm e}}^1(\mathbb{R}^d)$ is identified with
\begin{align*}
\{f \in L^2_{\text{\rm loc}}(\R^d,m) \mid |\nabla f| \in L^2(\mathbb{R}^d,m)\}.
\end{align*}
Here, $L^2_{\text{\rm loc}}(\R^d,m)$ is the space of locally square integrable functions on $\R^d$ with respect to $m$, and $\nabla f$ denotes the distributional gradient of $f$. Even if $d>2$, the extended Dirichlet space is characterized with distributional derivatives (\cite[Theorem~2.2.12]{CF}). From these facts and \cite[(5.2.17)]{CF}, we find that the Dirichlet form $(\mathcal{E},\mathcal{F}^\mu)$ is identified with
\begin{align}
\mathcal{E}(f,g)&=\frac{1}{2}\int_{\R^d} (\nabla f(x),\nabla g(x))\,dm(x),\quad f,g \in \mathcal{F}^\mu, \label{eq:diri}\\
\mathcal{F}^\mu&=\{\widetilde{u} \in H_{\text{\rm e}}^{1}(\mathbb{R}^d) \mid \widetilde{u} \in L^2(\mathbb{R}^d,\mu)\}. \notag 
\end{align}
Here, we denote by $(\cdot,\cdot)$ the standard inner product on $\mathbb{R}^d$, and $\widetilde{u}$ is the quasi-continuous version of $u \in H_{\text{\rm e}}^1(\mathbb{R}^d)$. See \cite[Lemma~2.1.4 and Theorem~2.1.7]{FOT} for the existence and the uniqueness. However, even in this setting, it is generally difficult to  write down other analytical objects associated with $B^\mu$, such as  the semigroup and the resolvent. Therefore, it is non-trivial to clarify how these objects depend on $\mu$.

 In this paper, we study the continuity of the resolvent of $B^\mu$ in the spatial component.
Even though this kind of problem can be formulated for other Markov processes, the current setting allows us to quantitatively clarify how the continuity depends on $\mu$. In the main theorem of this paper (Theorem~\ref{thm:rsf}), we prove that the resolvent of $B^\mu$ is H\"{o}lder continuous in the spatial component under a certain condition on  $\mu$. The condition is given in \eqref{eq:ballm} below, and the index there represents a regularity of $\mu$. This also describes a lower bound for the index of the H\"{o}lder continuity of the resolvent. In particular, we see that the  resolvent is $(1-\eps)$-H\"{o}lder continuous if the index is sufficiently large. Condition \eqref{eq:ballm} can be  regarded as a generalized concept of the $d$-measure.  We refer the reader to \cite{FU} for basic facts on time-changed Hunt processes by PCAFs associated with $d$-measures. We also note that the Liouville measure is one of examples which satisfies $\eqref{eq:ballm}$. The reader is referred to \cite{GRV0,GRV} and references therein for more details and the time changed planar Brownian motion by the PCAF associated with the Liouvllle measure.

If $F^\mu=\R^d$, it is not very hard to see that the resolvent of $B^\mu$ is just H\"{o}lder continuous.
In fact, we see from \eqref{eq:diri} that any bounded harmonic function $h$ on $B(z,2r)$ ($z \in \R^d,\,r>0$) with respect to $B^\mu$ is also harmonic with respect to the standard Brownian motion. Here, $B(z,r)$ denotes the open ball centered at $z$ with radius $r>0$. Then, from \cite[Chapter~II. (1.3)~Proposition]{B0}, there exists a positive constant $C$ independent of $z$ and $r$ such that
\begin{align*}
|h(x)-h(y)| \le C\sup_{z \in \mathbb{R}^d}|h(z)|  \frac{|x-y|}{r},\quad x,y \in B(z,r).
\end{align*}
Furthermore, since $A^\mu$ is a homeomorphism on $[0,\infty)$ (here, we used the assumption that $F^\mu=\R^d$), $A^\mu_{\tau_{B(x,r)}}$ is identified with the exit time of $B^\mu$ from $B(x,r)$, where $\tau_{B(x,r)}$ denotes the first exit time of $B$ from $B(x,r)$. This observation and the regularity condition~\eqref{eq:ballm} lead us to a mean exit time estimate for $B^\mu$ (see also Lemma~\ref{lem:pcaf}). Then, the same argument as in \cite[(5.6) Proposition, Section~VII]{B} shows that the resolvent of $B^\mu$ is H\"{o}lder continuous. We note that this kind of argument is applicable to several situations (see, e.g., \cite[Proposition~3.3 and Theorem~3.5]{BKK}); however,  even if $\mu=m$, this only implies that the index of the  H\"{o}lder continuity is greater than or equal to $2/3$. This estimate is not sharp because $B^m$ is the standard Brownian motion and the resolvent is Lipschitz continuous. Thus, even if $F^\mu=\R^d$, our result does not directly follow from the method stated above, and implies rather sharp result.
 
 For the proof of Theorem~\ref{thm:rsf}, we use the mirror coupling of $d$-dimensional Brownian motions. The key to our proof is an inductive argument based on the strong Markov property (of the coupling) and some estimates of the coupling time (Lemmas~\ref{lem:key} and \ref{lem:excp}). Since mirror couplings of stochastic processes are universal concepts, our arguments may be useful for estimating the indices of H\"{o}lder continuity of resolvents for other time-changed Markov processes.   

The remainder of this paper is organized as follows. In Section~2, we set up a framework and state the main theorem (Theorem~\ref{thm:rsf}). 
In Section~3, we provide some preliminary estimates for PCAFs of the $d$-dimensional Brownian motion. In Section~4, we introduce some lemmas on the mirror coupling of Brownian motions, and  prove Theorem~\ref{thm:rsf}.

\vspace{0.3cm}

\noindent
{\it Notation.} In the paper, we use the following symbols and conventions.
\begin{itemize}
\item $(\cdot,\cdot)$ and $|\cdot|$ denote the standard inner product and norm of $\R^d$, respectively.
\item For $x\in\R^d$ and $r>0$, $B(x,r)$ (resp.\ $\cl{B}(x,r)$) denotes the open (resp.\ closed) ball in $\R^d$ with center $x$ and radius $r$.
\item For a subset $S \subset \R^d$ and $f \colon S \to [-\infty,\infty]$, we set $\|f\|_\infty:=\|f\|_{\infty,S}:=\sup_{x \in S}|f(x)|$.
\item For a topological space $S$, we write $\mathcal{B}_b(S)$ for the space of bounded Borel measurable functions on $S$. 
\item For $a,b \in [-\infty,\infty]$, we write $a \vee b=\max\{a,b\}$ and $a \wg b=\min \{a,b\}$.
\item $\inf\emptyset=\infty$ by convention.
\end{itemize}

\section{Main results} 
Let $B=(\{B_t\}_{t \ge 0},\{P_x\}_{x \in \mathbb{R}^d})$ be a Brownian motion on $\R^d$. The Dirichlet form is identified with
\begin{align*}
\mathcal{E}(f,g)&=\frac{1}{2}\int_{\R^d}(\nabla f(x), \nabla g(x))\,dm(x),\quad
f,g \in H^1(\R^d).
\end{align*}
Here, $H^1(\R^d)(=H_{\text{\rm e}}^1(\mathbb{R}^d) \cap L^2(\R^d,m))$ denotes the first-order Sobolev space on $\R^d$. For an open subset $U \subset \R^d$ and for a subset $A \subset \R^d$, we define 
\begin{align*}
\text{cap}(U)&=\inf\left\{ \mathcal{E}(f,f)+\int_{\R^d}f^2\,dm \relmiddle|  f \in H^1(\R^d),\ f \ge 1,\, m\text{-a.e. on }U\right\},\\
\text{Cap}(A)&=\inf\left\{ \text{cap}(U)  \mid A \subset U\text{ and } U\text{ is an open subset of }\R^d\right\}.
\end{align*}
A non-negative Radon measure $\mu$ on $\R^d$ is said to be {\it smooth} if $\mu(A)=0$ for any $A \subset \R^d$ with $\text{Cap}(A)=0$. For a smooth measure $\mu$, by \cite[Theorem~4.1.1]{CF}, there exists a unique PCAF $A^\mu=\{A_t^{\mu}\}_{t \ge0}$ of $B$ such that for any non-negative functions $f,g \in \mathcal{B}_{b}(\R^d)$ and $\alpha>0$,
\begin{align}
&\int_{\R^d}E_{x}\left[\int_{0}^{\infty}e^{-\alpha t}f(B_t)\,dA_t^{\mu} \right]g(x)\,dm(x) \notag \\
&=\int_{\R^d}E_x\left[\int_{0}^{\infty}e^{-\alpha t}g(B_t)\,dt \right]f(x)\,d\mu(x),\label{eq:revuz}
\end{align}
where $E_x$ denotes the expectation under $P_x$. 
See \cite[Section~4]{CF} and \cite[Section~5]{FOT} for  the definition and further details on PCAFs. We also note that the exceptional set of $A^{\mu}$ can be taken to be empty (see \cite[Theorem~4.1.11]{CF}). 

Let $\{\tau_t^\mu\}_{t \ge 0}$ be the right continuous inverse of $A^\mu$. We define 
\begin{align*}
B_t^{\mu}&=B_{\tau_t^{\mu}},\ P_x^{\mu}=P_x,\quad (t ,x) \in [0,\infty) \times F^\mu,
\end{align*}
where $F^\mu$ denotes the support of $A^\mu$:
\begin{align}\label{eq:support}
 F^\mu=\{x \in \R^d \mid \inf\{t>0 \mid A_t^\mu>0\}=0,\, P_x\text{-a.s.}\}.
\end{align}
We note that $F^\mu$ is a nearly Borel subset with respect to $B$ (see the paragraph after \cite[(A.3.11)]{CF}). The support $F^\mu$ is also regarded as a topological subspace of $\R^d$.
By \cite[Theorems~5.2.1 and A.3.11]{CF}, $B^\mu=(\{B_t^{\mu}\}_{t \ge 0}, \{P_x^{\mu}\}_{x \in F^\mu})$ is a $\mu$-symmetric right process on $F^\mu$. The resolvent  $\{G_\alpha^\mu\}_{\alpha>0}$ is given by
\[
G_\alpha^\mu f(x)=E_{x}\left[\int_{0}^\infty e^{-\alpha t}f(B_t^\mu)\,dt \right],\quad \alpha>0,\,\,f \in \mathcal{B}_b(F^\mu), \,\, x \in F^\mu.
\]
Let $\mathcal{B}_b^{\ast}(\R^d)$ denote the space of bounded universally measurable functions on $\R^d$. That is, any $f \in \mathcal{B}_b^{\ast}(\R^d)$ is bounded and measurable with respect to the $\sigma$-field $\mathcal{B}^\ast (\R^d)$; the family of universally measurable subsets of $\R^d$: $\mathcal{B}^\ast (\R^d):=\bigcap_{\mu \in \mathcal{P}(\R^d)}\mathcal{B}^\mu (\R^d)$. Here, $\mathcal{P}(\R^d)$ denotes the family of all probability measures on $\R^d$ and $\mathcal{B}^\mu (\R^d)$ is the completion of the Borel $\sigma$-field $\mathcal{B}(\R^d)$ on $\R^d$ with respect to $\mu \in \mathcal{P}(\R^d).$
For $\alpha>0$, $f \in \mathcal{B}_b^\ast(\R^d)$, and $x \in \R^d$, we define
\[
V_{\alpha}^\mu f(x)=E_{x}\left[\int_{0}^{\infty}e^{-\alpha A_t^\mu} f(B_t)\,dA_t^\mu \right].
\]
We see from \cite[Exercise~A.1.29]{CF} that $F^\mu$ is a universally measurable subsets of $\R^d$. 
By noting this fact and  using \cite[Lemma~A.3.10]{CF}, we have
\begin{align}
G_\alpha^\mu f(x)=V_{\alpha}^\mu f(x).\label{eq:resol}
\end{align}
for any $\alpha>0$, $f \in \mathcal{B}_b(F^\mu)$, and $x \in F^\mu$.

Now we are in a position to state our main theorem.
\begin{theorem}\label{thm:rsf}
Let $p \in \R^d$ and assume that there exist $\kappa>d-2$, $R \in (0,1]$, and $K>0$ such that 
\begin{equation}
\mu(B(x,r)) \le K r^{\kappa} \label{eq:ballm}
\end{equation}
for any $r \le R$ and $x \in \R^d$ with $|x-p| \le r$.
 Then, for any $\alpha>0$ and $\eps \in (0, (2-d+\kappa)/(3-d+\kappa))$, there exists $C>0$ depending  on $d$, $p$, $\kappa$, $K$, $R$, $\eps$, and $\alpha$ such that 
\begin{equation}\label{eq:hol}
\left |V_{\alpha}^\mu f(x)-V_{\alpha}^\mu f(y) \right| \le C\|f\|_{\infty} |x-y|^{\{(2-d+\kappa) \wg 1\}-\eps} 
\end{equation}
for any $f \in \mathcal{B}_b^{\ast}(\R^d)$ and  $x,y\in B(p,2^{-C'}R)$, where $C'$ is a positive number depending on $d$, $\eps$ and $\kappa$. In particular, we have 
\begin{equation}\label{eq:hol*}
\left |G_\alpha^\mu f(x)-G_\alpha^\mu f(y) \right| \le C\|f\|_{\infty} |x-y|^{\{(2-d+\kappa) \wg 1\}-\eps} 
\end{equation}
for any $f \in \mathcal{B}_{b}(F^\mu)$ and  $x,y\in F^\mu \cap B(p,2^{-C'}R)$.
\end{theorem}

%We don't know whether the semigroup of $B^\mu$ possesses the same estimate as \eqref{eq:hol*}.

\section{Preliminary lemmas}

For an open subset $U \subset \R^d$, we denote by $U \cup \{\partial_U\}$ the one-point compactification. We set
$
\tau_U=\inf\{t \in [0,\infty) \mid B_t \notin U\}.
$ Then, the absorbing Brownian motion $B^{U}=(\{B_t^U\}_{t \ge0},\{P_x\}_{x \in U})$ on $U$ is defined as
\begin{align*}
B_{t}^{U}=
\begin{cases}
B_t ,\quad t<\tau_{U}, \\
\partial_{U},\quad t \ge \tau_{U}.
\end{cases}
\end{align*}
We write $p^U=p_t^U(x,y) \colon (0,\infty) \times U \times U \to [0,\infty)$ for the transition density of $B^U$. That is, $p^U$ is the jointly continuous function such that
\begin{align*}
P_x(B_t^ U \in dy)=p_{t}^{U}(x,y)\,dm(y),\quad t>0,\,x \in U.
\end{align*}
The Green function of $B^U$ is defined by 
\[
g_U(x,y)=\int_{0}^\infty p_t^U(x,y)\,dt,\quad x,y\in U.
\]

\begin{lemma}\label{lem:ppcaf}
Let $U \subset \R^d$ be an open subset. Then, for any $x \in U$, $t>0$, and non-negative $f \in \mathcal{B}_{b}(U)$,
\begin{align*}%\label{eq:ppcaf}
E_{x}\left[\int_{0}^{t \wg \tau_U}f(B_s)\,dA_s^{\mu} \right]=\int_{U}\left(\int_{0}^t p_{s}^{U}(x,y)\,ds \right)f(y)\,d\mu(y).
\end{align*}
In particular, we have
\begin{align*}
E_{x}\left[\int_{0}^{\tau_U}f(B_s)\,dA_s^{\mu} \right]=\int_{U}g_U(x,y)f(y)\,d\mu(y).
\end{align*}
\end{lemma}
\begin{proof}
We fix $t>0$ and non-negative functions $f,g \in \mathcal{B}_{b}(U)$. We may assume that $f$ is compactly supported.
By \cite[Proposition~4.1.10]{CF}, we have 
\begin{align}\label{eq:ppcaf2}
&\int_{U}E_{z}\left[ \int_{0}^{t \wg \tau_U}f(B_s)\,dA_s^{\mu} \right]g(z)\,dm(z)=\int_{0}^t \left(\int_{U}(P_{s}^{U}g)(x) f(x)\,d\mu(x)\right)\,ds.
\end{align}
We use Fubini's theorem to obtain that
\begin{align} 
&\int_{0}^t \left(\int_{U}(P_{s}^{U}g)(x) f(x)\,d\mu(x)\right)\,ds \notag \\
&=  \int_{U} \left(\int_{0}^t \left(\int_{U}p_{s}^{U}(x,z)f(x)\,d\mu(x) \right)\,ds \right) g(z)\,dm(z). \label{eq:ppcaf3}
\end{align}
Because $\mu$ is a Radon measure, by letting $g=\bone_{U}$ in \eqref{eq:ppcaf2} and \eqref{eq:ppcaf3}, we see that
\[ E_{(\cdot)}\left[ \int_{0}^{t \wg \tau_U}f(B_s)\,dA_s^{\mu} \right]   \quad\text{ and }\quad
\int_{0}^t \left(\int_{U}p_{s}^{U}(x,\cdot)f(x)\,d\mu(x) \right)\,ds
\]
are integrable on $U$ with respect to $m$. 
%Both sides of \eqref{eq:ppcaf2} and \eqref{eq:ppcaf3} are finite since $\mu$ is a Radon measure on $U$.
 Moreover, because $g$ is arbitrarily taken, \eqref{eq:ppcaf2} and \eqref{eq:ppcaf3} imply that for $m\text{-a.e. } z\in U$,
\begin{align}\label{eq:ppcaf4}
E_{z}\left[ \int_{0}^{t \wg \tau_U}f(B_s)\,dA_s^{\mu} \right] 
&=\int_{0}^t \left(\int_{U}p_{s}^{U}(x,z)f(x)\,d\mu(x) \right)\,ds.
\end{align}
By following the convention that $f(\partial_U)=0$, we see that the left-hand side of \eqref{eq:ppcaf4} is equal to $E_{z}[ \int_{0}^{t}f(B_s^{U})\,dA_{s \wg \tau_U}^{\mu}]$. Hence, we have
for $m\text{-a.e. } z\in U$,
\begin{align}\label{eq:ppcaf4-2}
E_{z}\left[ \int_{0}^{t}f(B_s^{U})\,dA_{s \wg \tau_U}^{\mu} \right]=\int_{0}^t \left(\int_{U}p_{s}^{U}(x,z)f(x)\,d\mu(x) \right)\,ds.
\end{align}
We see from \cite[Exercise~4.1.9~(iii)]{CF} that $\{A_{s \wg \tau_U}^\mu \}_{s \ge 0}$ is the PCAF of $B^U$. By using the additivity, the Markov property of $B^U$, and \eqref{eq:ppcaf4-2}, we obtain that for any $x \in U$,
\begin{align*}
&E_{x}\left[ \int_{0}^{t \wg \tau_U}f(B_s)\,dA_s^{\mu} \right]=\lim_{u \downarrow 0}E_{x}\left[ \int_{u}^{t}f(B_s^{U})\,dA_{s \wg \tau_U}^{\mu} \right]\\
&=\lim_{u \downarrow 0}E_{x}\left[E_{B_u^U} \left[\int_{0}^{t-u}f(B_s^{U})\,dA_{s \wg \tau_U}^{\mu} \right]\right]\\
&=\lim_{u \downarrow 0 }\int_{U}p_{u}^{U}(x,z) \left(\int_{0}^{t-u} \left(\int_{U}p_{s}^{U}(y,z)f(y)\,d\mu(y) \right)\,ds\right)\,dm(z) \\
&=\lim_{u \downarrow 0 } \int_{u}^{t} \left(\int_{U}p_{s}^{U}(x,y)f(y)\,d\mu(y) \right)\,ds=\int_{0}^{t} \left(\int_{U}p_{s}^{U}(x,y)f(y)\,d\mu(y) \right)\,ds,
\end{align*}
which completes the proof.  ``In particular'' part immediately follows from the monotone convergence theorem. \qed
\end{proof}

Let $d \ge 2,$ $r \in (0,1)$, and $x,y \in \R^d$. Then, by \cite[Example~1.5.1]{FOT}, 
\begin{align}\label{eq:green1}
g_{B(x,r)}(x,y)&=
\left\{
    \begin{alignedat}{2}
&-\frac{1}{\pi} \log|x-y|, &\quad d=2,\\
&\frac{\Gamma(d/2-1)}{2\pi^{d/2}}  |x-y|^{2-d},&\quad d \ge 3.
         \end{alignedat}
    \right. 
\end{align}
Here, $\Gamma$ denotes the gamma function. If $d=1$, we see from \cite[Lemma~20.10]{K} that for any $a,b \in \R$ with $a<b$,
\begin{align}\label{eq:green2}
g_{(a,b)}(x,y) =\frac{2(x \wedge y-a)(b-x \vee y)}{b-a},\quad x,y \in (a,b).
\end{align}
For $s \in (0,1]$ and $t>0$, we define 
\begin{align*}
\zt_d(s,t)=
\begin{cases}
s^{2-d+t},&\quad d\ge 3\text{ or }d=1,\\
-s^{t}\log s,&\quad d=2.\\
\end{cases}
\end{align*}
\begin{lemma}\label{lem:pcaf}
Let $p \in \R^d$ and take constants $\kappa >d-2$, $R \in (0,1]$, and $K>0$ so that \eqref{eq:ballm} holds. 
Then, there exists $C \in (0,\infty)$ depending on $d$, $p$, $\kappa$, $R$,  and $K$ such that for any $r \in (0, R/2]$ and $x \in B(p,r)$
\begin{align*}
&\int_{B(x,r)}g_{B(x,r)}(x,y)\,d\mu(y) \le C \zeta_d(r,\kappa).
\end{align*}
In particular,  we have $
E_{x}[A_{\tau_{B(x,r)}}^{\mu}]\le   C \zeta_d(r,\kappa).$
\end{lemma}
\begin{proof}
In view of \eqref{eq:ballm}, we have for any $r \in (0, R/2]$ and $x \in  B(p,r)$,
\begin{align}
\mu( B(x,r)) \le K r^{\kappa}.\label{eq:ballm2}
\end{align}
Therefore, when $d=1$, we use \eqref{eq:green2} to obtain that 
\begin{align*}
\int_{B(x,r)}g_{B(x,r)}(x,y)\,d\mu(y) \le 4 Kr^{\kappa+1}.
\end{align*}
Equation~\eqref{eq:green1} implies that for any $k\in \N$,  
\begin{align*}
\sup_{y \in \R^d \setminus B(x,r2^{-k})}g_{B(x,r)}(x,y) \le
\left\{
    \begin{alignedat}{2}
&-\frac{1}{\pi}\log(r2^{-k}), &\quad d=2,\\
&\frac{\Gamma(d/2-1)}{2\pi^{d/2}} (r2^{-k})^{2-d}, &\quad d\ge 3.
         \end{alignedat}
    \right. 
\end{align*}
Thus, for $d=2$, we obtain from \eqref{eq:ballm2} that
\begin{align*}
\int_{B(x,r)}  g_{B(x,r)}(x,y)\,d\mu(y) 
&=  \sum_{k=1}^{\infty}\int_{B(x,r2^{-(k-1)}) \setminus B(x,r2^{-k}) }g_{B(x,r)}(x,y)\,d\mu(y) \\
&\le  -\frac{K}{\pi} \sum_{k=1}^{\infty}  (r2^{-(k-1)})^{\kappa}\log(r2^{-k})\le -Cr^{\kappa} \log r.
 \end{align*}
Here, $C$ is a positive constant depending on $\kappa$ and $K$. 
For $d\ge 3$, we similarly use \eqref{eq:ballm2} to obtain that
\begin{align*}
\int_{B(x,r)}  g_{B(x,r)}(x,y)\,d\mu(y) &\le   \frac{K\Gamma(d/2-1)}{2\pi^{d/2}}  \sum_{k=1}^{\infty} (r2^{-(k-1)})^{\kappa }  (r2^{-k})^{2-d}\\
&\le \left(\frac{K\Gamma(d/2-1)2^{\kappa }}{2\pi^{d/2}}   \sum_{k=1}^\infty 2^{-\{(2-d)+\kappa \}k}\right) r^{2-d+\kappa }.
 \end{align*}
 Because $\kappa >d-2$, we have \[
\int_{B(x,r)}  g_{B(x,r)}(x,y)\,d\mu(y)  \le Cr^{2-d+\kappa}.\]
Here, $C$ is a positive constant depending on $d$, $K$ and $\kappa.$ ``In particular'' part immediately follows from Lemma~\ref{lem:ppcaf}. \qed
\end{proof}

\section{Proof of Theorem~\ref{thm:rsf}}

 Let $x,y \in \R^d$ and $\{W_t\}_{t \ge 0}$ a $d$-dimensional Brownian motion starting at the origin. The mirror coupling $(Z^x,\widetilde{Z}^y)=(\{Z_t^x\}_{t\ge 0},\{\widetilde{Z}_t^y\}_{t \ge 0})$ of $d$-dimensional Brownian motions is defined as follows:
\begin{itemize}
\item For any $t< \inf\{s>0 \mid Z_s^x=\widetilde{Z}_s^y\}$,
\begin{align}
Z_t^x&=x+W_t, \quad \notag \\
\widetilde{Z}_t^y&=y+W_t-2 \int_{0}^{t} \frac{Z_s^x-\widetilde{Z}_s^y}{|Z_s^x-\widetilde{Z}_s^y|^2} ( Z_s^x-\widetilde{Z}_s^y, \,dW_s ).\label{eq:mir}
\end{align}
\item For any $t \ge \inf\{s>0 \mid Z_s^x=\widetilde{Z}_s^y\}$, we have $Z_t^x=\widetilde{Z}_t^y$.
\end{itemize} 
\begin{remark}\label{rem:cp}
\begin{enumerate}
\item[(1)] The mirror coupling $(Z^x,\widetilde{Z}^y)$ is a special case of couplings for diffusion processes studied in \cite[Section~3]{LR}. 
\item[(2)] For $x,y \in \R^d$ with $x \neq y$ and  $t<\inf\{s>0 \mid Z_s^x=\widetilde{Z}_s^y\}$, we have
\[
Z_t^{x}-\widetilde{Z}_{t}^{y}=x-y+2\int_{0}^{t} \frac{Z_s^x-\widetilde{Z}_s^y}{|Z_s^x-\widetilde{Z}_s^y|^2} ( Z_s^x-\widetilde{Z}_s^y, \,dW_s ).
\]
This implies that the random vector $Z_{t}^{x}-Z_{t}^{y}$ is parallel to $x-y$. We then see from \eqref{eq:mir} that $\widetilde{Z}_t^y$ coincides with the mirror image of $Z_{t}^{x}$ with respect to the hyperplane $H_{x,y}=\{z \in \mathbb{R}^d \mid (z-(x+y)/2,x-y)=0\}$. Further, $\inf\{s>0 \mid Z_s^x=\widetilde{Z}_s^y\}=\inf\{s>0 \mid Z_s^x \in H_{x,y}\}$. Then, it is easy to see that $(Z^x,\widetilde{Z}^y)$ is a strong Markov process on $\R^d \times \R^d.$
\end{enumerate}
\end{remark}
For $x,y \in \R^d$, we define $\xi_{x,y}=\inf\{t>0 \mid Z_t^x=Z_t^y\}.$ We denote by $P_{x,y}$ the distribution of $(Z^x,\widetilde{Z}^y)$. For $t \ge 0$, we set 
\begin{align*}
A_{t}^{\mu,x}&=A_{t}^{\mu}(Z^x),\quad \widetilde{A}_{t}^{\mu,y}=A_{t}^{\mu}(\widetilde{Z}^y ),
\end{align*}
where we regard  $A_t^\mu$ as $[0,\infty]$-valued functions on $C([0,\infty),\R^d)$, the space of $\R^d$-valued continuous functions on $[0,\infty)$. Then, $\{A_{t}^{\mu,x}\}_{t \ge 0}$ and $\{\widetilde{A}_{t}^{\mu,y}\}_{t \ge 0}$
become PCAFs of $Z^x$ and $\widetilde{Z}^y$, respectively. Furthermore, $A^{\mu,x}$ and $\widetilde{A}^{\mu,y}$ can be regarded as PCAFs of the coupled process $(Z^x,\widetilde{Z}^y)$ in the natural way.

For $x,y \in \R^d$, we define 
\begin{align}\label{eq:eqci}
\mathcal{I}_{x,y}=E_{x,y}\bigl[ A_{\xi_{x,y}}^{\mu,x} \wg 1 \bigr], \quad
\widetilde{\mathcal{I}}_{x,y}=E_{x,y}\bigl[ \widetilde{A}_{\xi_{x,y}}^{\mu,y} \wg 1 \bigr],
\end{align}
where $E_{x,y}$ denotes the expectation under $P_{x,y}$. 
At the end of this section (see \eqref{eq:rsf4} below), we will show that for any  $f \in \mathcal{B}_b^{\ast}(\R^d)$ with $\|f\|_{\infty}\le 1$,
\begin{align*}
|V_{\alpha}^\mu f(x)-V_{\alpha}^\mu f(y)| \le 2(1+\alpha^{-1}) (\mathcal{I}_{x,y}+\widetilde{\mathcal{I}}_{x,y}),\quad \alpha>0,\, x,y \in \R^d.
\end{align*}
We now introduce some lemmas to estimate the expectations in \eqref{eq:eqci}.

\begin{lemma}\label{lem:key}
Let $x,y \in \R^d$, and $\tau$ be a stopping time of $(Z^x,\widetilde{Z}^y)$. Then,
\begin{align*}
E_{x,y}\left[\left| Z_{ \xi_{x,y} \wg \tau }^{x} -\widetilde{Z}_{\xi_{x,y} \wg \tau}^{y} \right|^\theta \right] \le |x-y|^{\theta} 
\end{align*}
for any $\theta \in (0,1]$.
\end{lemma}
\begin{proof}
We fix $t \ge0$ and  $x,y \in \R^d$  with $x \neq y$. To simplify the notation, we write $Z_t$ (resp. $\widetilde{Z}_t$, $L_t$, $\widetilde{L}_t$, $\xi$) for $Z_t^x$  (resp. $Z_t^y$, $L_t^x$, $\widetilde{L}_t^y$, $\xi_{x,y}$).  We also fix $n \in \N$ such that $|x-y| \ge 1/n$, and set $\xi_n=\inf\{s>0 \mid |Z_s-\widetilde{Z}_s| \le 1/n\}$. 

For $s<\xi_n \wg \tau$, we define
\begin{align*}
\alpha_s=(Z_s-\widetilde{Z}_s)(Z_s-\widetilde{Z}_s)^{T}/|Z_s-\widetilde{Z}_s|^2.
\end{align*}
Here, $(Z_s-\widetilde{Z}_s)^{T}$ denotes the transpose of $Z_s-\widetilde{Z}_s.$ From  It\^{o} formula, 
\begin{align*}
&|Z_{t \wg  \xi_n \wg \tau}-\widetilde{Z}_{t \wg  \xi_n \wg \tau}|^2-|x-y|^2=2\int_{0}^{t \wg  \xi_n \wg \tau}(Z_s-\widetilde{Z}_s, \alpha_s\,dW_s )+t \wg  \xi_n \wg \tau
\end{align*}
and for any $\theta \in (0,1]$,
\begin{align} \label{eq:cp1}
&|Z_{t \wg  \xi_n \wg \tau}-\widetilde{Z}_{t \wg  \xi_n \wg \tau}|^{\theta}-|x-y|^{\theta} \notag\\
&=\theta \int_{0}^{t \wg  \xi_n \wg \tau} |Z_s-\widetilde{Z}_s|^{\theta-2}(Z_s-\widetilde{Z}_s, \alpha_s\,dW_s )\notag \\
&\quad+\{\theta/2+\theta\left(\theta/2-1 \right)\}\int_{0}^{t \wg  \xi_n \wg \tau} |Z_s-\widetilde{Z}_s|^{\theta-2}\,ds.
\end{align}
Since the first term above is a martingale and the second one is non-positive, by taking the expectations of both sides of \eqref{eq:cp1}, we arrive at 
\begin{align}\label{eq:cp2}
E_{x,y}\left[\left|Z_{t \wg  \xi_n \wg \tau}-\widetilde{Z}_{t \wg  \xi_n \wg \tau}\right|^{\theta}\right] \le |x-y|^{\theta}.
\end{align}
Letting $n \to \infty$ in \eqref{eq:cp2}, we complete the proof.\qed
\end{proof}

\begin{lemma}\label{lem:cp}
It holds that
$$P_{x,y}(t<\xi_{x,y}) \le |x-y|/\sqrt{2\pi t}$$
for any $t>0$ and $ x,y \in \R^d$.
\end{lemma}
\begin{proof}
Let $t \ge0$ and  $x,y \in \R^d$  with $x \neq y$. We take $n \in \N$ such that $|x-y| \ge 1/n$.
Letting $\theta=1$ in \eqref{eq:cp1}, we have 
\begin{align}\label{eq:cp6}
&|Z_{t \wg  \xi_n }-\widetilde{Z}_{t \wg  \xi_n}|-|x-y| =\int_{0}^{t \wg  \xi_n} |Z_s-\widetilde{Z}_s|^{-1}(Z_s-\widetilde{Z}_s, \alpha_s\,dW_s ).
\end{align}
The quadratic variation of the right-hand side of \eqref{eq:cp6} equals to $t \wg \xi_n$, $t \ge 0$. Hence, by the Dambis--Dubins--Schwartz theorem, there is a one-dimensional Brownian motion $\beta=\{\beta_s\}_{s \ge 0}$ such that 
\begin{align}\label{eq:cp7}
\beta_{t \wg \xi_n}= \int_{0}^{t \wg  \xi_n} |Z_s-\widetilde{Z}_s|^{-1}(Z_s-\widetilde{Z}_s, \alpha_s\,dW_s ).
\end{align}
By using \eqref{eq:cp6}, \eqref{eq:cp7}, and the reflection principle of the Brownian motion, we have
\begin{align}\label{eq:cp8}
P_{x,y}(\xi_n>t) &\le P_{x,y} \left( -|x-y| \le \inf_{0 \le s \le t} \beta_s \right)  \notag \\
&=1-2\int_{-\infty}^{-|x-y|} \frac{1}{\sqrt{2 \pi t}}\exp(-u^2/2t)\,du \notag  \\
&=\int_{-|x-y|}^{|x-y|}\frac{1}{\sqrt{2 \pi t}}\exp(-u^2/2t)\,du \le |x-y|/\sqrt{2\pi t}. 
\end{align}
Letting $n \to \infty$ in \eqref{eq:cp8} completes the proof. \qed
\end{proof}

For $x,y \in \R^d$ and an open subset $U \subset \R^d$, we define
\begin{align*}
\tau_{U}^x&=\tau_{U}(Z^x),\quad \widetilde{\tau}_{U}^y=\tau_{U}(\widetilde{Z}^y)
\end{align*}
where we regard  $\tau_{U}$ as $[0,\infty]$-valued function on $C([0,\infty),\R^d)$. We note that $\tau_U^x=\inf\{t>0 \mid (Z_t^x,\widetilde{Z}_t^y) \notin U \times \R^d\}$. Hence, $\tau_U^x$ and $\widetilde{\tau}_U^y$ are exit times of  $(Z^x,\widetilde{Z}^y)$. We also see from \cite[Lemma~II.1.2]{S} that there exists $C>0$ depending on $d$ such that 
\begin{align}\label{ref:exit}
P_{x,y}(\tau_{B(x,r)}^x \le t) \le C
\exp(-r^2/Ct)
\end{align}
for any $(x,y,t) \in \R^d \times \R^d \times (0,\infty)$ and $r \in (0,\infty)$.
\begin{lemma}\label{lem:excp}
Let $\chi,\eps \in (0,1]$, $R>0$ and  $n \ge 1$ be positive numbers. Then, there is a positive constant $C$ depending on $\eps$, $R$, and $n$ such that 
\begin{align*}
&P_{x,y}(\tau_{B(x,2^{-n}R|x-y|^{\chi}) }^x  \le \xi_{x,y})  \le C|x-y|^{1-\chi-\eps}
\end{align*}
for any $x,y \in \R^d$ with $|x-y| \in (0,1]$.
\end{lemma}
\begin{proof}
We fix $\chi,\eps \in (0,1]$, $R>0$, $n \ge 1$. Let $x,y \in \R^d$ with $|x-y| \in (0,1]$. By \eqref{ref:exit} and Lemma~\ref{lem:cp}, there exists $C>0$ such that for any $t>0$,
\begin{align}
&P_{x,y}(\tau_{B(x,2^{-n}R|x-y|^{\chi}) }^x \le \xi_{x,y}) \notag \\
&\le P_{x,y}(\tau_{B(x,2^{-n}R|x-y|^{\chi})}^x \le t)+P_{x,y}(\xi_{x,y}>t) \notag \\
&\le C
\exp(-2^{-2n}R^2|x-y|^{2\chi}/Ct)+|x-y|/\sqrt{2\pi t}. \label{eq:excp3}
\end{align}
Then, letting $t=|x-y|^{2\chi+2\eps} (\in (0,1])$ in \eqref{eq:excp3}, we have
\begin{align}
&P_{x,y}(\tau_{B(x,2^{-n}R|x-y|^{\chi}) }^x \le \xi_{x,y}) \notag \\
&\le C
\exp(-2^{-2n}R^2|x-y|^{-2\eps}/C)+|x-y|^{1-\chi-\eps}.\label{eq:excp4} 
\end{align}
For any $\dl \in (0,1]$ and $c \in (0,\infty)$ there exists $c_{\dl} \in (0,\infty)$ depending on $\dl$ and $c$ such that for any $r \in [0,\infty)$,
\begin{align}
\exp(-cr^{-\dl}) \le c_{\dl}r \label{eq:m2}.
\end{align}
By using \eqref{eq:excp4} and \eqref{eq:m2}, we obtain the desired inequality. \qed
\end{proof}

From now on, we fix $p \in \R^d$, and take constants $\kappa>d-2$, $R \in (0,1]$, and $K>0$ so that  \eqref{eq:ballm} holds. Theorem~\ref{thm:rsf} is proved by an inductive argument. The following lemma is the first step. 

\begin{lemma}\label{lem:ind1}
Let $\eps \in (0, (2-d+\kappa)/(3-d+\kappa))$. There exists $C>0$ depending on $d$, $\eps$, $p$, $\kappa$, $R$, and $K$  such that
\begin{align*}
 \mathcal{I}_{x,y} &\le C|x-y|^{(2-d+\kappa)/(3-d+\kappa)-\eps}
 \end{align*}
  for any $x,y \in  B(p, R/2)$.

\end{lemma}
\begin{proof}
Let $\eps \in (0, (2-d+\kappa)/(3-d+\kappa))$.  We fix $x,y \in B(p, R /2)$ and set
\[r=R|x-y|^\chi/2,\]
where $\chi \in (0,1]$ is a positive number which will be chosen later. Because $|x-y| \le 1$, we have $r \le R/2$.
A straightforward calculation gives
\begin{align}\label{eq:func1}
\mathcal{I}_{x,y}&\le E_{x,y}\left[A_{\tau_{B(x,r)}^x}^{\mu,x}  \right]+ P_{x,y}(\tau_{ B(x,r)}^x \le \xi_{x,y}) .
\end{align}
By applying Lemmas~\ref{lem:pcaf} and \ref{lem:excp} to \eqref{eq:func1}, we obtain that 
\begin{align}\label{eq:func3}
\mathcal{I}_{x,y}&\le C\{ \zeta_d(r,\kappa)+|x-y|^{1-\chi-\eps}\}.
\end{align}
Here, $C>0$  is a positive constant depending on $d$, $\eps$, $\chi$, $p$, $\kappa$, $R$,  and $K$.

Next, we  optimize the right-hand side of \eqref{eq:func3} in $\chi$. 
Note that we have for any $a,b >0$ with $b\le a$, 
\begin{align}\label{eq:log}
-s^a \log s \le (1/b)s^{a-b},\quad s \in (0,1].
\end{align}
Thus, if $d=2$, we have
\begin{align*}
\zeta_d(r,\kappa) & \le \frac{\chi}{\eps} \left(\frac{R}{2}\right)^{\kappa-(\eps/\chi)}  |x-y|^{\kappa \chi-\eps}
\end{align*}
provided that  $\eps /\chi \le \kappa$. Let $\chi$ be the solution to $\kappa \chi-\eps=1-\chi-\eps$. Then, $\chi=1/(\kappa+1)$. Further, $\eps/\chi \le \kappa$ and  
\begin{align*}
\mathcal{I}_{x,y}&\le C'|x-y|^{\frac{\kappa}{\kappa+1}-\eps},
\end{align*}
where $C'$ is a positive constant depending on $\eps$, $\chi$, $p$, $\kappa$, $R$,  and $K$. 

If $d \ge 3$ or $d=1$,  we have 
\begin{align*}
\zeta_d(r,\kappa) &\le r^{2-d+\kappa-(\eps/\chi)}=\left(\frac{R|x-y|^\chi}{2}\right)^{2-d+\kappa-(\eps/\chi)} \\
&=\left(\frac{R}{2}\right)^{2-d+\kappa-(\eps/\chi)}  |x-y|^{\chi(2-d+\kappa)-\eps}.
\end{align*}
Let $\chi$ be the solution to $\chi(2-d+\kappa)-\eps=1-\chi-\eps$. Then, $\chi=1/(3-d+\kappa) \in (0,1]$ and 
\begin{align*}
\mathcal{I}_{x,y}&\le C''|x-y|^{(2-d+\kappa)/(3-d+\kappa)-\eps}.
\end{align*}
Here, $C''$ is a positive constant depending on $d$, $\eps$, $\chi$, $p$, $\kappa$, $R$,  and $K$. 
\qed
\end{proof}

For $\eps \in (0,1)$ and $n \in \N$, we set 
\begin{align*}
q_{n,\kappa, \eps}=\frac{r_n-\eps r_{n-1}}{r_n+1}-\eps,
\end{align*}
where $r_n$ is a positive number defined by 
\begin{align*}
r_n=(2-d+\kappa)  (r_{n-1}+1), \quad r_0=0.
\end{align*}
We then find that $r_n=\sum_{l=1}^n (2-d+\kappa)^l$ and $\{r_n\}_{n=1}^\infty$ is increasing.
For any $n \in \N$, 
\begin{align} 
q_{n,\kappa,\eps}>0 &\Longleftrightarrow \frac{r_n}{r_n+r_{n-1}+1}>\eps \notag \\
&\Longleftrightarrow \frac{(2-d+\kappa) r_n}{(2-d+\kappa) r_n+(2-d+\kappa)(r_{n-1}+1)}>\eps \notag \\
&\Longleftrightarrow \frac{2-d+\kappa}{3-d+\kappa}>\eps \label{eq:rel2}.
\end{align}

\begin{lemma}\label{lem:ind}
Let $n \in \N$ and $\eps \in (0, (2-d+\kappa)/(3-d+\kappa))$. Then, there exists a positive constant  $C_1$ depending on $d$, $\eps$, $p$, $\kappa$, $R$, $K$, and $n$ such that 
\begin{align}
 \mathcal{I}_{x,y} &\le C_1|x-y|^{q_{n,\kappa, \eps}} \label{eq:ind0}
\end{align}
for any $x,y \in B(p,2^{-n}R)$. 
\end{lemma}
\begin{proof}
 If $n=1$, the conclusion follows from Lemma~\ref{lem:ind1}. In what follows, we suppose that \eqref{eq:ind0} holds for some $n \in \N$. Then, for any $\eps \in (0, (2-d+\kappa)/(3-d+\kappa))$, there exists $C_1>0$ depending on $d$, $\eps$, $p$, $\kappa$, $R$, $K$, and $n$ such that 
\begin{align}
 \mathcal{I}_{x,y} &\le C_1 |x-y|^{q_{n,\kappa, \eps}} \label{eq:indn}
\end{align}
for any $x,y \in B(p,2^{-n}R)$. 

Let $\chi \in (0,1]$ be a positive number which will be chosen later. We fix $\eps \in (0,(2-d+\kappa)/(3-d+\kappa))$, and $x,y \in B(p,2^{-n-1})$. To simplify the notation, we write 
\begin{align*}
\tau&=\tau_{B(x,2^{-n-1}R|x-y|^{\chi})}^x,\quad
 \widetilde{\tau}=\widetilde{\tau}_{B(y,2^{-n-1}R|x-y|^{\chi})}^y, \quad
 \xi=\xi_{x,y}.
 \end{align*}
In view of Remark~\ref{rem:cp}~(2), we have $\tau=\widetilde{\tau}$. It is straightforward to show that  
\begin{align}
\mathcal{I}_{x,y}&\le E_{x,y}\bigl[ A_{ \tau}^{\mu,x} \wg 1 : \xi \le \tau \bigr] +E_{x,y}\bigl[ A_{ \tau }^{\mu,x} \wg 1:  \tau <\xi \bigr] \notag  \\
&\quad+E_{x,y}\bigl[ (A_{\xi}^{\mu,x}-A_{\tau}^{\mu,x} )\wg 1 : \tau  < \xi \bigr] \notag \\
&= E_{x,y}\bigl[ A_{ \tau}^{\mu,x} \wg 1 \bigr]+E_{x,y}\bigl[ (A_{\xi}^{\mu,x}-A_{\tau }^{\mu,x} )\wg 1 : \tau < \xi \bigr]\notag \\
&=:  %E_{x,y}\bigl[ A_{ \tau}^{\mu,x}  \bigr]+E_{x,y}\bigl[ (A_{\xi}^{\mu,x}-A_{\tau \wg \widetilde{\tau}}^{\mu,x} )\wg 1 : \tau \wg \widetilde{\tau} < \xi \bigr] \notag \\ &
I_1+I_2.\label{eq:indn1} 
\end{align}
On the event $\{\xi > \tau \}$, we have $
A_{\xi}^{\mu,x}-A_{\tau }^{\mu,x}=A_{\xi-\tau}^{\mu,x} \circ \theta_{\tau}\le A_{\xi}^{\mu,x} \circ \theta_{\tau},$
where $\{\theta_t\}_{t \ge 0}$ denotes the shift operator of the coupled process $(Z^x,\widetilde{Z}^y)$. We know from Remark~\ref{rem:cp}~(2) that $(Z^x,\widetilde{Z}^y)$ is a strong Markov process on $\R^d \times \R^d$. Therefore, we obtain that 
\begin{align}
&I_2=E_{x,y}\left[ \left( A_{\xi-\tau}^{\mu,x} \circ \theta_{\tau }\right) \wg 1: \tau <  \xi \right] \notag \\
&\le E_{x,y}\left[ E_{Z_{\tau }^x, \widetilde{Z}_{\tau}^y} 
\left[A_{\xi}^{\mu,x} \wg 1\right]   : \tau <  \xi\right] =E_{x,y}\left[ \mathcal{I}_{Z_{\tau}^x, \widetilde{Z}_{ \tau}^y}   : \tau <  \xi\right]. \label{eq: indn3}
\end{align}
Observe that $Z_{\tau}^x \in \overline{B}(x,2^{-n-1}R)$ and $ \widetilde{Z}_{ \tau}^y=\widetilde{Z}_{\widetilde{\tau}}^y \in \overline{B}(y,2^{-n-1}R)$. Furthermore, by noting that $x,y \in  B(p,2^{-n-1}R)$, we have $|p-Z_{\tau}^x|<2^{-n}R$ and $|p-\widetilde{Z}_{\widetilde{\tau}}^y|<2^{-n}R$. Then, we use \eqref{eq:indn} to obtain that
\begin{align}
&E_{x,y}\left[ \mathcal{I}_{Z_{\tau}^x, \widetilde{Z}_{\tau}^y}   : \tau <  \xi\right]\le C_1E_{x,y}\left[ \left| Z_{\tau}^x-\widetilde{Z}_{\tau}^y \right|^{q_{n,\kappa,\eps}}: \tau <  \xi \right]. \label{eq: indn4}
\end{align}
Let $a_n=(r_n+1)/r_n$ and $b_n=r_n+1$, Then,   $a_n^{-1}+b_n^{-1}=1$. Because $\eps \in (0,(2-d+\kappa)/(3-d+\kappa))$,  \eqref{eq:rel2} implies that $0<a_{n}q_{n,\kappa,\eps}$ and
\[
a_{n}q_{n,\kappa,\eps}=\frac{r_n+1}{r_n} \left( \frac{r_n- \eps r_{n-1}}{r_n+1}-\eps \right)\le 1.
\]
By using H\"{o}lder's inequality, Lemmas~\ref{lem:key} and \ref{lem:excp}, we obtain that
\begin{align}
&E_{x,y}\left[ \bigl| Z_{\tau}^x- \widetilde{Z}_{\tau }^y \bigr|^{q_{n,\kappa,\eps}}  :\tau  <  \xi \right] \notag \\
&\le  E_{x,y} \left[ \bigl| Z_{\tau \wg  \xi }-\widetilde{Z}_{ \tau \wg  \xi}\bigl|^{a_n q_{n,\kappa, \eps}} \right]^{1/a_n}  P_{x,y}( \tau  <  \xi)^{1/b_n} \notag \\
&\le   |x-y|^{a_n q_{n,\kappa, \eps}/a_n} P_{x,y}( \tau<  \xi)^{1/b_n} \notag \\
&\le  C_2|x-y|^{q_{n,\kappa, \eps}+(1-\chi-\eps)/b_n} \label{eq: indn5} ,
\end{align}
where $C_2$ is a positive constant  depending on $\eps$, $R$, and $n$. Therefore, \eqref{eq: indn3}, \eqref{eq: indn4}, and \eqref{eq: indn5} imply
\begin{align}
I_2 \le C_3|x-y|^{q_{n,\kappa, \eps}+(1-\chi-\eps)/b_n} \label{eq: indn5.1} .
\end{align}
Here, $C_3$ is a positive constant depending on $d$, $\eps$, $p$, $\kappa$, $R$, $K$, and $n$.

On the other hand, Lemma~\ref{lem:pcaf} yields 
\begin{align}\label{eq:indn2}
I_1 \le E_{x,y}\bigl[ A_{\tau}^{\mu,x}  \bigr]
& \le  C_4 \zeta_d(2^{-n-1}R|x-y|^{\chi},\kappa),
\end{align}
where $C_4$ is a positive constant depending on $d$, $p$, $\kappa$, $R$,  and $K$. If $d=2$, we use \eqref{eq:log} to obtain that
\begin{align}\label{eq:indn2.0}
\zeta_d(2^{-n-1}R|x-y|^{\chi},\kappa)  &\le (\chi/\eps)  (2^{-n-1}R)^{\kappa-(\eps/\chi)} |x-y|^{\kappa \chi-\eps} .
\end{align}
provided that  $\eps /\chi \le \kappa$. If $d \ge 3$ or $d=1$, 
\begin{align}\label{eq:indn2.1}
\zeta_d(2^{-n-1}R|x-y|^{\chi},\kappa) &=(2^{-n-1}R|x-y|^{\chi})^{2-d+\kappa} \notag \\
&\le (2^{-n-1}R)^{2-d+\kappa}  |x-y|^{\chi(2-d+\kappa)-\eps}.
\end{align}
Therefore, if $\eps \le (2-d+\kappa)\chi$, regardless of the value of $d$, we obtain from \eqref{eq:indn2}, \eqref{eq:indn2.0}, and \eqref{eq:indn2.1} that 
\begin{align}\label{eq: indn2.2}
I_1
& \le  C_5 |x-y|^{(2-d+\kappa) \chi-\eps}.
\end{align}
Here, $C_5$ is a positive constant depending on $d$, $\eps$, $p$, $\kappa$, $R$, $K$, and $n$.

Let $\eta$ be the solution to 
\begin{align*}
(2-d+\kappa) \eta-\eps &=q_{n,\kappa, \eps}+(1-\eta-\eps)/b_n. 
\end{align*} 
Then, a direct calculation and the definition of $b_n$ imply that
\begin{align*}
\eta&=\frac{b_n q_{n,\kappa,\eps}+1+\eps (b_n-1)}{b_n(2-d+\kappa)+1}=\frac{(r_n+1)q_{n,\kappa,\eps}+1+\eps r_n}{(2-d+\kappa)(r_n+1)+1}.
\end{align*}
From the relation $r_{n+1}=(2-d+\kappa) (r_n+1)$ and the definition of $q_{n,\kappa,\eps}$,
\begin{align*}
&(2-d+\kappa) \eta-\eps \notag \\
&=\frac{(2-d+\kappa) (r_n+1)q_{n,\kappa,\eps}+(2-d+\kappa)(1+\eps r_n)}{r_{n+1}+1}-\eps \notag \\
%&=\frac{(2-d+\kappa) \{(r_n-\eps r_{n-1})-\eps(r_n+1)\}+(2-d+\kappa)(1+\eps r_n)}{r_{n+1}+1} -\eps \notag \\
&=\frac{r_{n+1}-\eps r_n}{r_{n+1}+1}-\eps=q_{n+1,\kappa, \eps}.
\end{align*} 
Because $q_{n+1,\kappa, \eps}>0$, we have $(2-d+\kappa) \eta>\eps.$
Noting the fact that $\{r_n\}_{n=1}^{\infty}$ is increasing, we have
\begin{align*}
0<\eta=\frac{r_{n+1}  - \eps r_n}{(2-d+\kappa)(r_{n+1}+1)} \le \frac{r_{n+1}}{(2-d+\kappa)(r_{n+1}+1)}=\frac{r_{n+1}}{r_{n+2}} \le 1.
\end{align*}
Therefore, we can set $\chi=\eta$. By combining \eqref{eq:indn1}, \eqref{eq: indn5.1}, and \eqref{eq: indn2.2}, we see \eqref{eq:ind0} holds for $n+1$. \qed
\end{proof}

Because $\{r_n \}_{n=1}^\infty$ is increasing, we have for any $n \in \N$ and $\eps \in (0,1)$,
\begin{align*}
q_{n,\kappa, \eps} = \frac{r_n-\eps r_{n-1}}{r_{n}+1}-\eps \ge \frac{r_n}{r_{n}+1}-2\eps .
\end{align*}
If $2-d+\kappa \ge 1$, $\lim_{n \to \infty}r_n=\infty$. If $2-d+\kappa \in (0,1)$, 
\[
\lim_{n \to \infty}r_n=\frac{2-d+\kappa}{1-(2-d+\kappa)}.
\] Therefore, we obtain that $
\varliminf_{n \to \infty} q_{n,\kappa, \eps}  \ge (2-d+\kappa) \wedge 1 -2\eps.$ Since the same estimate as Lemma~\ref{lem:ind} holds for $\widetilde{ \mathcal{I}}_{x,y}$, we have the following  corollary.
\begin{corollary}\label{cor:lim}
For any $\eps \in (0, (2-d+\kappa)/(3-d+\kappa))$, there exists a positive constants  $C$ depending on $d$, $\eps$, $p$, $\kappa$, $R$, and $K$ such that 
\begin{align*}
 \mathcal{I}_{x,y}+\widetilde{ \mathcal{I}}_{x,y} &\le C|x-y|^{(2-d+\kappa) \wedge 1 -\eps}
\end{align*}
for any $x,y \in  B(p,2^{-C'}R)$, where $C'$ is a positive constant depending on $d$, $\eps$ and $\kappa$.
\end{corollary}

We now prove Theorem~\ref{thm:rsf}.
\begin{proof}[of Theorem~\ref{thm:rsf}]
Let $\alpha>0$, $x,y \in \R^d$ and $f \in\mathcal{B}_b^{\ast}(\R^d)$.  Without loss of generality, we may assume that $\|f\|_{\infty}\le1$. We write $\xi=\xi_{x,y}$ for the simplicity. Then, we have
\begin{align}
 &V_{\alpha }^{\mu}f(x)- V_{\alpha}^{\mu}f(y)   \notag \\
 &= E_{x,y}\left[ \int_{\xi}^{\infty}\exp(-\alpha A_t^{\mu,x})f(Z_t^x)\,dA_t^{\mu,x}\right] 
 - E_{x,y}\left[ \int_{\xi}^{\infty}\exp(-\alpha \widetilde{A}_t^{\mu,y})f(\widetilde{Z}_t^y)\,d\widetilde{A}_t^{\mu,y}\right]  \notag \\
&\quad+ E_{x,y}\left[ \int_{0}^{\xi}\exp(-\alpha A_t^{\mu,x})f(Z_t^x)\,dA_t^{\mu,x}\right]- E_{x,y}\left[ \int_{0}^{\xi}\exp(-\alpha \widetilde{A}_t^{\mu,y})f(\widetilde{Z}_t^y)\,d\widetilde{A}_t^{\mu,y}\right] \notag \\
&=:J_1-J_2+J_3-J_4.  \label{eq:rsf} 
\end{align}
Because $\{A_{t}^{\mu,x}\}_{t \ge 0}$ is a PCAF of $(Z^x,\widetilde{Z}^y)$, we have $A_{t+\xi}^{\mu,x}=A_{\xi}^{\mu,x}+A_{t}^{\mu,x}\circ \theta_{\xi}$ and $dA_{t+\xi}^{\mu,x}=dA_{t}^{\mu,x}\circ \theta_{\xi}$. By using  these equations and the strong Markov property of $(Z^x,\widetilde{Z}^y)$, we obtain that
\begin{align*} 
J_1&=E_{x,y}\left[ \int_{0}^{\infty}\exp(-\alpha A_{t+\xi}^{\mu,x})f(Z_{t+\xi}^x)\,dA_{t+\xi}^{\mu,x}\right]  %&=E_{x,y}\left[ \exp(-A_{\xi}^{\mu,x})\int_{0}^{\infty}\exp(- A_{t}^{\mu,x} \circ \theta_\xi)f(\text{proj}^1(Z_{t}^x,\widetilde{Z}_{t}^y) \circ \theta_{\xi})\,dA_{t}^{\mu,x} \circ \theta_{\xi}\right] \notag \\
=E_{x,y}\left[ \exp(-\alpha A_{\xi}^{\mu,x})V_{\alpha}^{\mu}f(Z^x_{\xi})\right],\\
J_2&=E_{x,y}\left[\exp(-\alpha \widetilde{A}_{\xi}^{\mu,y}) V_{\alpha}^{\mu}f(\widetilde{Z}^y_{\xi})\right].
\end{align*}
Since $Z_{\xi}^x=\widetilde{Z}_{\xi}^y$, we have
\begin{align*}
J_1-J_2  &=E_{x,y}\left[ \exp(-\alpha A_{\xi}^{\mu,x})V_{\alpha}^{\mu}f(Z^x_{\xi})\right] -E_{x,y}\left[\exp(-\alpha A_{\xi}^{\mu,x}) V_{\alpha}^{\mu}f(\widetilde{Z}^y_{\xi})\right]  \\
&\quad +E_{x,y}\left[\exp(-\alpha A_{\xi}^{\mu,x}) V_{\alpha}^{\mu}f(\widetilde{Z}^y_{\xi})\right] -E_{x,y}\left[\exp(-\alpha \widetilde{A}_{\xi}^{\mu,y}) V_{\alpha}^{\mu}f(\widetilde{Z}^y_{\xi})\right]   \notag \\
&=0+E_{x,y}\left[\left\{ \exp(- \alpha A_{\xi}^{\mu,x})-\exp(-\alpha \widetilde{A}_{\xi}^{\mu,y}) \right\}V_{\alpha}^{\mu}f(\widetilde{Z}^y_{\xi})\right] \notag.
\end{align*}
Because $ |\alpha V_{\alpha}^{\mu}f(\widetilde{Z}^y_{\xi})| \le \|f\|_{\infty}=1$ and the function $s\mapsto e^{-\alpha s}$ is $\alpha$-Lipschitz continuous on $[0,\infty)$, we obtain that
\begin{align} \label{eq:rsf2-1} 
|J_1-J_2| &\le  E_{x,y}\left[\bigl| A_{\xi}^x-\widetilde{A}_{\xi}^y \bigr| \wg \alpha^{-1} \right] \notag \\
%&\le E_{x,y}\left[ A_{\xi}^x \wg 1\right]+ E_{x,y}\left[ \widetilde{A}_{\xi}^y \wg 1\right]
&\le (1+\alpha^{-1})  (\mathcal{I}_{x,y}+\widetilde{\mathcal{I}}_{x,y}).
\end{align}
From Jensen's inequality,
\begin{align}
|J_3-J_4| &\le \alpha^{-1} E_{x,y}\bigl[1-\exp(- \alpha A_{\xi}^x)\bigr]+ \alpha^{-1} E_{x,y}\bigl[1-\exp(-\alpha \widetilde{A}_{\xi}^y)\bigr] \notag \\
&\le  E_{x,y}\left[ A_{\xi}^x \wg \alpha^{-1}\right]+E_{x,y}\left[ \widetilde{A}_{\xi}^y \wg \alpha^{-1}\right] \notag \\
&\le (1+\alpha^{-1}) (\mathcal{I}_{x,y}+\widetilde{\mathcal{I}}_{x,y}). \label{eq:rsf3} 
\end{align}
By using \eqref{eq:rsf}, \eqref{eq:rsf2-1}, and \eqref{eq:rsf3}, we arrive at
\begin{align}\label{eq:rsf4}
\left|V_{\alpha }^{\mu}f(x)- V_{\alpha}^{\mu}f(y)  \right| &\le 2 (1+\alpha^{-1}) \bigl( \mathcal{I}_{x,y}+\widetilde{\mathcal{I}}_{x,y} \bigr).
\end{align}
Corollary~\ref{cor:lim} and \eqref{eq:rsf4} yield the desired estimate.
``In particular'' part immediately follows from \eqref{eq:resol}.
  \qed
\end{proof}

%\begin{thebibliography}{99}

%\bibitem{revuz-yor}
%    D.\ Revuz, M.\ Yor: \emph{Continuous Martingales and Brownian Motion}.
%    Springer, Berlin 1999 (3rd ed).

%\end{thebibliography}

\end{document}